\documentclass[twoside,12pt]{amsart}
\usepackage{latexsym}
\usepackage{amssymb}
\usepackage{amsthm}
\usepackage{cases}
\usepackage{epsfig}
\usepackage{multirow}
\usepackage{caption}
\usepackage{subfigure}
\usepackage{graphicx}
\usepackage{color}

\usepackage[colorinlistoftodos, textsize=tiny]{todonotes}
\allowdisplaybreaks 
\newtheorem{example}{\textbf{Example}}[section]
\numberwithin{equation}{section}

\newtheorem{thm}{Theorem}[section]

\theoremstyle{definition}

\theoremstyle{remark}

\def\bq{\begin{equation}}
\def\eq{\end{equation}}
\begin{document}
\title[RFEM for biharmonic equation]{Recovery based finite element method for
biharmonic equation in two dimensional}
\author[Y. Huang, H. Wei, W. Yang and N. Yi]{Yunqing Huang$^\dagger$,  Huayi Wei$^\dagger$, Wei Yang$^\dagger$, and  Nianyu Yi$^{\dagger, *}$}
\thanks{$^*$ Corresponding author.}
\address{$^\dagger$ Hunan Key Laboratory for Computation and Simulation in Science and Engineering; School of Mathematics and Computational Science, Xiangtan University, Xiangtan 411105, P.R.China}
\email{huangyq@xtu.edu.cn;\ weihuayi@xtu.edu.cn;\ yangwei@xtu.edu.cn;\ yinianyu@xtu.edu.cn}


\begin{abstract} 
    We design and numerically validate a recovery based linear finite element method 
    for solving the biharmonic equation. The main idea is to replace the gradient operator
    $\nabla$ on linear finite element space by $G(\nabla)$ in the weak
    formulation of the biharmonic equation, where $G$ is the recovery
    operator which recovers the piecewise constant function into the linear
    finite element space.  By operator $G$, Laplace operator $\Delta$ is
    replaced by $\nabla\cdot G(\nabla)$. Furthermore the boundary condition on
    normal derivative $\nabla u\cdot \pmb{n}$ is treated by the boundary penalty
    method. The explicit matrix expression of the proposed method is also
    introduced.  Numerical examples on uniform and adaptive meshes are
    presented to illustrate the correctness and effectiveness of the proposed
    method.  
\end{abstract}
\maketitle

%\tableofcontents

\section{Introduction}\label{sec:Int}

The biharmonic equation is a fourth order equation which arises in areas of
continuum mechanics, including linear elasticity theory and the solution of
Stokes flow.  In this work, we consider a $C^0$ linear
finite element method for the biharmonic equation in two-dimensional space. 

\bq\label{1e1}
\Delta^2u(x,y) = f(x,y), \qquad \forall (x,y)\in \Omega, 
\eq
with boundary conditions 

\bq\label{1e2}u(x,y)=g_1(x,y),\qquad (x,y)\in\partial\Omega,\eq
\bq\label{1e3}u_n(x,y)=g_2(x,y),\qquad (x,y)\in\partial\Omega.\eq
Here $\Omega$ is a bounded domain in the two-dimensional space $\mathbb{R}^2$
with a Lipschitz boundary $\partial\Omega$, $u_n=\nabla u\cdot \pmb{n}$ is the normal
derivative of $u$ on $\partial\Omega$, and $\pmb{n}$ is the unit normal vector
pointing outward. The biharmonic operator $\Delta^2$ is defined through

\[
    \Delta^2=\nabla^4=\frac{\partial^4}{\partial x^4}+2\frac{\partial^4}{\partial
x^2\partial y^2}+\frac{\partial^4}{\partial y^4}.
\] 
The basic idea of our method is applying the gradient recovery technique as
pre-processing tool to solve the high-order partial differential equations. 

The mixed form is rewrite the biharmonic equaiton \eqref{1e1}-\eqref{1e3} into a coupled system of Poisson equations as  
\bq\label{1e4}\left\{\begin{aligned}
\Delta v(x,y)=f(x,y),&\qquad (x,y)\in \Omega,\\
\Delta u(x,y)=v(x,y),&\qquad (x,y)\in \Omega,\\
u(x,y)=g_1(x,y),&\qquad (x,y)\in\partial\Omega,\\
u_n(x,y)=g_2(x,y),&\qquad (x,y)\in\partial\Omega.
\end{aligned}\right.
\eq
One can easily see that under this formulation, there are two boundary
conditions for the solutions $u$ but no boundary condition for the new variable
$v$. Thus, it is much more difficult to solve the biharmonic equation with the
boundary conditions \eqref{1e2} and \eqref{1e3}. These computations are dependent on accurate
evaluation of the missing boundary values for $v$, and the computational
procedures are often unsatisfactory. The treatment of the boundary condition
for the splitting method is a challenging problem since poor boundary
approximations may reduce the accuracy of  the numerical solution.  An alternative
technique is the so-called coupled equation approach,
\[\begin{aligned}
&\left\{\begin{aligned}
&\Delta v(x,y)=f(x,y),\qquad (x,y)\in \Omega,\\
&v(x,y)=\Delta u(x,y)-c(u_n-g_2(x,y)),\qquad (x,y)\in\partial \Omega,
\end{aligned}\right.\\
&\left\{\begin{aligned}
&\Delta u(x,y)=v(x,y),\qquad (x,y)\in \Omega,\\
&u(x,y)=g_1(x,y),\qquad (x,y)\in\partial \Omega,
\end{aligned}\right.
\end{aligned}\]
where $c$ is a constant, see \cite{E71, M74}. For a given initial guess
$v_0(x,y)$, an iteration solution $(u_k(x,y), v_k(x,y))$ can be computed
until its convergence.
%Our approach is advantageous because the given Dirichlet boundary conditions are exactly satisfied and no approximations need to be carried out at the boundaries, in contrary to the splitting method.

There are various finite element methods to discretize the biharmonic equation
in the literature. As the most classical approach, the $C^1$ conforming finite
element methods require the basis functions and their derivatives are
continuous on $\bar{\Omega}$, which are rarely used in practice for their too
many degrees of freedom and implementation complexity.  For example, the
Argyris finite element method \cite{C91} has $21$ degrees of freedom for
triangles. The nonconforming finite element methods such as the Adini
element or Morley element \cite{BL04, C91, WX06} are popular methods for
numerical solution of the high-order partial differential equations.  The key
idea in nonconforming methods is to use the penalty term to ensure the
convergence into the natural energy space of the variational problem.  Mixed
finite element method is another choice which is based on the equivalent form
\eqref{1e4} and only require the Lagrangian finite element spaces, which are
widely used in practice, but they require very careful treatment on the
essential and natural boundary conditions. The literature on the mixed finite
element methods is vast, and we refer to \cite{BG11, BF91, GNP08, W04} and the
references therein for the detail of these methods. The discontinuous Galerkin
method is also a choice which is based on standard continuous Lagrangian finite
element spaces \cite{BS05, EGHLMT02} or completely discontinuous finite element
spaces \cite{GH09, SM07}.  Other methods which have been developed for fourth
order problems include finite difference methods \cite{BCCF09, CLL08, GP79},
and finite volume method \cite{EGHL12}.

An alternative to aforementioned methods is the recovery based finite element
method developed in recent years \cite{CGZZ16, GZZ17, L15}. It is a
nonconforming finite element method based on the discretization of the Laplace
operator defined by applying the gradient recovery operator on the gradient of
the $C^0$ linear element. The variational formulation of \eqref{1e1} involves
the term $(\Delta u, \Delta v)$. The idea in this paper is to redefine the
discrete gradient operator and furthermore the Laplace operator involved in the
weak formulation, by embedding a gradient recovery operator in pre-processing,
such that the linear finite element can be used for solving the biharmonic
equation. The resulting finite element scheme is state as follows:

\bq \label{1es} 
\int_{\Omega}\nabla\cdot G(\nabla u_h)\nabla\cdot G(\nabla
v_h)dx+\frac{\sigma}{h^2}\int_{\Gamma}G(\nabla u_h)\cdot \pmb{n}G(\nabla v_h)\cdot
\pmb{n}ds=\int_{\Omega}fvdx+\frac{\sigma}{h^2}\int_{\partial\Omega}g_2G(\nabla v_h)\cdot
\pmb{n}ds,
\eq
where gradient recovery operator $G$ is embedded in a priori way such that the
$\nabla\cdot G(\nabla v_h)$ is well-defined for any function $v_h\in V_h$. The
boundary condition \eqref{1e3} is incorporating in the finite element scheme by
a penalty method. Notice that the difference between our scheme \eqref{1es} and
the existing recovery based finite element scheme for biharmonic equation
\cite{CGZZ16, GZZ17, L15} is the treatment of boundary condition \eqref{1e3},
especially for the non-homogeneous boundary data. In \cite{CGZZ16, GZZ17, L15},
the boundary condition \eqref{1e3} is treated as an essential boundary
condition, that is enforcing the numerical solution satisfies the boundary
condition \eqref{1e3}. While in our scheme \eqref{1es}, we impose the boundary
condition \eqref{1e3} using the boundary penalty method. In this paper, we
develop and numerically investigate the recovery based finite element method
\eqref{1es} for biharmonic equation.

The remaining parts of this paper are organized as follows. In Section
\ref{sec:RFEM}, we introduce the gradient recovery operator and then present a
recovery based linear finite element method for the biharmonic equation. In
Section \ref{sec:IMP}, we discuss the implementation issue. And in the
following Section \ref{sec:NE}, we present some numerical experiments to show
the correctness and effectiveness of our method. Finally,  we make some
concluding remarks in Section \ref{sec:CON}. 

\section{Recovery based finite element method}\label{sec:RFEM}
Consider the biharmonic equation
\[
\Delta^2u(x,y)=f(x,y), \qquad \forall (x,y)\in \Omega=(0, 1)^2, 
\]
with boundary conditions 
\[u(x,y)=0,\quad \nabla u(x,y)\cdot \pmb{n}=g(x,y),\quad \rm{on}\ \partial\Omega\]

In weak form, this problem reads: Find $u\in V^g$ such that
\[a(u,v)=L(v)\quad \forall v\in V^0,\] where
\[V^g=\{v\in H^1(\Omega): \nabla v\in H(div), v|_{\partial\Omega}=0, \quad \nabla u\cdot \pmb{n}|_{\partial\Omega}=g\},\]
\[V^0=\{v\in H^1(\Omega): \nabla v\in H(div), v|_{\partial\Omega}= \nabla v\cdot \pmb{n}|_{\partial\Omega}=0\},\]
\[a(u,v)=\int_{\Omega}\nabla\cdot(\nabla u)\nabla\cdot(\nabla v)dx,\] and \[L(v)=\int_{\Omega}fvdx.\]

\subsection{Discrete spaces} 
Let $\mathcal{T}_h$ be a triangular partition of $\Omega\in \mathbb{R}^2$ with
mesh size $h$, and $h_{\tau}:=\text{diam}(\tau)$ for each element
$\tau\in\mathcal{T}_h$. We denote the set of vertices and edges of $\mathcal{T}_h$
by $\mathcal {N}_h$ and $\mathcal{E}_h$, respectively. The length of $E\in
\mathcal{E}_h$ is denoted by $h_E=\text{diam}(e)$. For each $E\in \mathcal{E}_h$,
denote a unit vector normal to $E$ by $n_E$, and $\omega_E$ denotes the union
of all elements that share $E$.  On each element $\tau\in\mathcal{T}_h$,
$P_k(\tau)$
denotes the polynomials on $\tau$ of degree $\leq k$.  Consider the $C^0$ linear
finite element space $S_h$ associated with $\mathcal {T}_h$ and defined by

\[
    S_h=\{v\in H^1(\Omega):v\in P_1(\tau), \forall \tau\in\mathcal{T}_h\}=\mathrm{span}\{\phi_z:
z\in \mathcal{N}_h\}.
\]  
The node basis functions of $S_h$ are the standard
Lagrangian basis functions.  The element patch is defined by
$\omega_z=\text{supp}\phi_z$. Furthermore, the peicewise constant function
space is denoted as 
\[\begin{aligned}
W_h:=\{w_h \in (L^{\infty}(\Omega))^2: w_h|_\tau\in (P_0(\tau))^2, \forall \tau\in\mathcal{T}_h\}.
\end{aligned}\]

%We first introduce the recovery operators which are used to define the discrete Lapace operator. 
\subsection{Recovery operator} 
In this subsection, we introduce the recovery operator which can recover a
piecewise constant function into the continuous piecewise linear finite element
space.  For simplicity, we take the weighted averaging recovery operator $G:
W_h\rightarrow S_h\times S^h$, which is defined as follows: for $w_h \in W_h$,

\bq\label{2e1}
G(w_h):=\sum\limits_{z\in\mathcal{N}_h}G(w_h)(z)\phi_z,\qquad
G(w_h)(z):=\sum\limits_{\tau \in\omega_z}w_\tau w_h|_\tau,
\eq
where the weights can be choosen as following \cite{HJY12}
\bq
\label{w1}{\rm Simple\ averaging:}\qquad w_\tau=\frac{1}{\sharp \omega_z},
\eq
\bq
\label{w2}{\rm Harmonic\ averaging:}\qquad
w_T=\frac{1/|\tau|}{\sum\limits_{\tau\in \omega_z}1/|\tau|}.
\eq 

Given $u_h\in S_h$, its gradient $\nabla u_h$ is piecewise constant and may
discontinuous across each element, thus $\Delta u_h$ is not well-defined. To
fix this problem, we use the recovery operator $G$ to 'lift' the gradient
$\nabla u_h$ into a vector finite element space in which $\nabla\cdot G(\nabla
u_h)$ is well-defined. In other words, we define the discrete Laplace operator
by $\Delta u_h:=\nabla\cdot G(\nabla u_h)$ for piecewise linear function
$u_h\in S_h$, where 
\[
G(\nabla u_h) = (G(\partial_x u_h), G(\partial_y u_h))^T.
\]

\subsection{Recovery based linear finite element scheme}
After defining the finite element spaces and the gradient recovery operators,
we now introduce the recovery based finite element method with a penalty for the
biharmonic equation. Let
\[S_h^0=S_h\cap H_0^1(\Omega)=\{v_h\in S_h: v_h|_{\partial\Omega}=0\}.\] 
The recovery based finite element scheme is to find $u_h\in S_h^0$ such that 
\bq\label{2e3}
a_h(u_h,v_h)=\int_{\Omega}fv_hdx+\frac{\sigma}{h^2}\int_{\Gamma}g_2G(\nabla v_h)\cdot \pmb{n}ds,\quad \forall v_h\in S_h^0,\eq
where 
\[a_h(u_h, v_h):=\int_{\Omega}\nabla\cdot G(\nabla u_h)\nabla\cdot G(\nabla v_h)dx+\frac{\sigma}{h^2}\int_{\Gamma}G(\nabla u_h)\cdot \pmb{n}G(\nabla v_h)\cdot \pmb{n}ds,\]
and the recovery operator $G$ is defined in \eqref{2e1}. Notice that the boundary
conditions \eqref{1e2} and \eqref{1e3} are treated in different ways. The
boundary condition \eqref{1e2} is treated as an essential boundary condition,
while the boundary condition \eqref{1e3} is imposed weakly with a boundary
penalty term in the discrete scheme.

\begin{thm}\label{thm1}
For the recovery based linear finite element scheme \eqref{2e3}, there exists a unique solution $u_h\in S_h^0$.
\end{thm}

\begin{proof}
Based on the scheme \eqref{2e3}, we define the following functional:
\bq\label{2e3a}\begin{aligned}
J(u_h):=&\frac{1}{2}\int_{\Omega}\left(\nabla\cdot G(\nabla u_h)\right)^2dx+\frac{\sigma}{2h^2}\int_{\Gamma}\left(G(\nabla u_h)\cdot \pmb{n}\right)^2ds\\
&-\int_{\Omega}fu_hdx-\frac{\sigma}{h^2}\int_{\Gamma}g_2G(\nabla u_h)\cdot \pmb{n}ds.
\end{aligned}\eq
Notice that the first and second terms of $J(u_h)$ are convex, and the third and fourth terms of $J(u_h)$ are linear with respect to $u_h$, then the functional $J(u_h)$ is a convex functional. 
Take the derivative of the functional $J(u_h)$, and for any $v_h\in S_h^0$ we have 
\[
\left(\frac{\delta J(u_h)}{\delta u_h}, v_h\right)=a_h(u_h, v_h)-\int_{\Omega}fv_hdx-\frac{\sigma}{h^2}\int_{\Gamma}g_2G(\nabla v_h)\cdot \pmb{n}ds=0.
\]
Then the uniqueness of the solution of scheme \eqref{2e3} is approved.
\end{proof}

\section{Implementation}\label{sec:IMP} 
In this section, we discuss the implementation of the term $(\nabla\cdot
G(\nabla u_h), \nabla\cdot G(\nabla v_h))$ in details, the calculation of the
other terms in recovery based fintie element scheme \eqref{2e3} are similar.

For simplicity, we only take the simple averaging (the weights are chosen as
\eqref{w1}) for illustration. For a mesh node $z_i\in \mathcal{N}_h$, let
$\phi_i$ denotes the basis function at node $z_i$, $\omega_i$ denotes the
element patch of $z_i$, and $\mathcal{N}(i)$ denotes the mesh nodes in
$\omega_i$. Then 
\[V_h=\text{span}\{\phi_i\}_{i=1}^N, \quad N=\sharp \mathcal{N}_h,\]
and 
\[u_h=\sum\limits_{i=1}^N=[\phi_1,\cdots, \phi_N]U, \qquad U=\begin{bmatrix}
u_1\\ u_2\\ \vdots\\ u_{N}
\end{bmatrix}.\]
From \eqref{2e1}, we have 
\[\begin{aligned}G(\nabla u_h)&=\begin{pmatrix}G(\partial_x u_h)\\
        G(\partial_y
    u_h)\end{pmatrix}=\begin{pmatrix}\sum\limits_{j=1}^NG(\partial_x u_h)(z_j)\phi_j\\
        \sum\limits_{j=1}^NG(\partial_y u_h)(z_j)\phi_j\end{pmatrix}=\begin{pmatrix}[\phi_1,\ \cdots,\ \phi_{N}]AU\\ 
[\phi_1,\ \cdots,\ \phi_{N}]BU\end{pmatrix},\\
G(\nabla \phi_i)&=\begin{pmatrix}G(\partial_x\phi_i)\\
G(\partial_y \phi_i)\end{pmatrix}=\begin{pmatrix}[C_{i,1},\cdots, C_{i,N}][\phi_1,\ \cdots,\ \phi_{N}]^T\\ 
[D_{i,1},\cdots, D_{i,N}][\phi_1,\ \cdots,\ \phi_{N}]^T\end{pmatrix},
\end{aligned}\]
where
\[
A_{i,j}=\left\{\begin{aligned}\sum\limits_{\tau\in \omega_i}\frac{1}{\sharp
        \omega_i}\partial_x\phi_j|_\tau,\quad &{\rm if}\ j=i\in \mathcal{N}(i),\\
\sum\limits_{\tau\in \omega_i\cap \omega_j}\frac{1}{\sharp
\omega_i}\partial_x\phi_j|_\tau,\quad &{\rm if}\ j\neq i\in \mathcal{N}(i),\\
0,\qquad &{\rm if}\ j\notin \mathcal{N}(i),
\end{aligned}\right.\]
\[
B_{i,j}=\left\{\begin{aligned}\sum\limits_{\tau\in \omega_i}\frac{1}{\sharp
    \omega_i}\partial_y\phi_j|_\tau,\quad &{\rm if}\ j=i\in \mathcal{N}(i),\\
\sum\limits_{\tau\in \omega_i\cap \omega_j}\frac{1}{\sharp
\omega_i}\partial_y\phi_j|_\tau,\quad &{\rm if}\ j\neq i\in \mathcal{N}(i),\\
0,\qquad &{\rm if}\ j\notin \mathcal{N}(i),
\end{aligned}\right.\]
\[
C_{i,j}=\left\{\begin{aligned}\sum\limits_{\tau\in \omega_j}\frac{1}{\sharp
    \omega_j}\partial_x\phi_i|_\tau,\quad &{\rm if}\ j=i\in \mathcal{N}(i),\\
\sum\limits_{\tau\in \omega_j\cap \omega_i}\frac{1}{\sharp
\omega_j}\partial_x\phi_i|_\tau,\quad &{\rm if}\ j\neq i\in \mathcal{N}(i),\\
0,\qquad &{\rm if}\ j\notin \mathcal{N}(i),
\end{aligned}\right.\]
\[
D_{i,j}=\left\{\begin{aligned}\sum\limits_{\tau\in \omega_j}\frac{1}{\sharp
    \omega_j}\partial_y\phi_i|_\tau,\quad &{\rm if}\ j=i\in \mathcal{N}(i),\\
\sum\limits_{\tau\in \omega_j\cap \omega_i}\frac{1}{\sharp
\omega_j}\partial_y\phi_i|_\tau,\quad &{\rm if}\ j\neq i\in \mathcal{N}(i).\\
0,\qquad &{\rm if}\ j\notin \mathcal{N}(i).
\end{aligned}\right.
\]

By taking $v_h=\phi_i, i=1,\cdots, N$, in matrix form, we obtain 
\[\begin{aligned}
\left(\nabla\cdot G(\nabla u_h), \nabla\cdot G(\nabla v_h)\right)
=&\left(\partial_xG(\partial_x u_h)+\partial_yG(\partial_y u_h),
\partial_xG(\partial_x v_h)+\partial_yG(\partial_y v_h)\right)\\
=&\left([\partial_x\phi_{1},\ \cdots,\ \partial_x\phi_{N}]AU+[\partial_y
\phi_{1},\ \cdots,\ \partial_y\phi_{N}]BU,\right.\\
&\left.\qquad C[\partial_x \phi_{1},\ \cdots,\
\partial_x\phi_{N}]^T+D[\partial_y\phi_{1},\ \cdots,\ \partial_y\phi_{N}]^T\right)\\
=&(CPA+CQB+DSA+DTB)U.\end{aligned}\]
where the matrices are calculated as following
\[P=\int_{\Omega}\begin{bmatrix}
\partial_x\phi_{1}\\\partial_x \phi_{2}\\ \vdots \\\partial_x \phi_{N}
\end{bmatrix}[\partial_x\phi_{1},\ \partial_x\phi_{2},\ \cdots,\ \partial_x\phi_{N}]dxdy,\]
\[Q=\int_{\Omega}\begin{bmatrix}
\partial_x\phi_{1}\\\partial_x \phi_{2}\\ \vdots\\ \partial_x \phi_{N}
\end{bmatrix}[\partial_y\phi_{1},\ \partial_y\phi_{2},\ \cdots,\ \partial_y\phi_{N}]dxdy,\]
\[S=\int_{\Omega}\begin{bmatrix}
\partial_y \phi_{1}\\\partial_y  \phi_{2}\\ \vdots\\ \partial_y  \phi_{N}
\end{bmatrix}[\partial_x\phi_{1},\ \partial_x \phi_{2},\ \cdots,\ \partial_x\phi_{N}]dxdy,\]
\[T=\int_{\Omega}\begin{bmatrix}
\partial_y \phi_{1}\\\partial_y \phi_{2}\\ \vdots\\ \partial_y \phi_{N}
\end{bmatrix}[\partial_y \phi_{1},\ \partial_y\phi_{2},\ \cdots\ ,\ \partial_y \phi_{N}]dxdy.\]

%\subsection{Algorithm}

\section{Numerical examples}\label{sec:NE}
In this section, we present some numerical examples to demonstrate the
performance of the recovery based linear finite element for the biharmonic
equation presented in \eqref{2e3}. We investigate the proposed recovery based
finite element method on the uniform regular mesh and the Centroidal
Voronoi-Delaunay Triangulation (CVDT) mesh. Also, we are interesting the
performance of the recovery based finite element method on adaptive meshes when
the solution of biharmonic equation appears singularity.

\begin{example}\label{exm1}
We first consider the biharmonic equation with homogeneous boundary condtions
\bq\label{4e1}\left\{
\begin{aligned}
&\Delta^2u(x,y)=f(x,y), \qquad \forall (x,y)\in \Omega=(0, 1)^2, \\
&u=0,\qquad \nabla u\cdot \pmb{n}=0,\quad \rm{on}\ \partial\Omega.
\end{aligned}\right.\eq
The exact solution is chosen the following function: \[u=\sin^2(\pi x)\sin^2(\pi y).\] 
Hence we choose $f=\Delta^2u$ as the function defined by
\[\begin{aligned}
f(x,y)=&8\pi^4(\sin^2(\pi x)-\cos^2(\pi x))\sin^2(\pi y)+8\pi^4\sin^2(\pi x)(\sin^2(\pi y)-\cos^2(\pi y))\\
&+8\pi^4(\sin^2(\pi x)-\cos^2(\pi x))(\sin^2(\pi y)-\cos^2(\pi y)).
\end{aligned}\]
\end{example}

The errors $\|u-u_h\|$, $\|\nabla u- \nabla u_h\|$, $\|\nabla u- G(\nabla
u_h)\|$, $\|\Delta u-\nabla\cdot G(\nabla u_h)\|$ and corresponding rates of
convergence are reported in Table \ref{tabexm11} and Table \ref{tabexm12}.
Table \ref{tabexm11} shows the numerical results on the uniform mesh in regular
pattern and Table \ref{tabexm12} shows the numerical results on the CVDT mesh.
We see clearly that: i) The $L^2$ errors $\|u-u_h\|$ and the gradient errors
$\|\nabla u- \nabla u_h\|$ converge at the rate of second order and first
order, respectively, which are optimal for the linear approximation; ii) The
recovered gradient $G(\nabla u_h)$ converges to the exact gradient $\nabla u$
under the second order rate, and one order higher than the gradient of the finite
element approximation. This shows that the recovered gradient is
superclose to the exact one; iii) The convergence rate of the error $\|\Delta u-\nabla\cdot
G(\nabla u_h)\|$ is first order.

\begin{table}[!htp]%\tabcolsep0.1in
\caption{Example \ref{exm1}, regular mesh, errors and convergence rates}\label{tabexm11}
\begin{tabular}[c]{|c|c|c|c|c|}
\hline
Dof &   441 &  1681 &  6561 & 25921\\\hline
$\| u - u_h\|$ &  0.01394 &  0.00344 &  0.00086 &  0.00021\\\hline
Order & -- &  2.02 &  2.00   &  2.00  \\\hline
$\|\nabla u - \nabla u_h\|$ &  0.39714 &  0.19032 &  0.09431 &  0.0471 \\\hline
Order & -- &  1.06 &  1.01 &  1.00  \\\hline
$\|\nabla u - G(\nabla u_h)\|$ &  0.0418  &  0.01051 &  0.00264 &  0.00066\\\hline
Order & -- &  1.99 &  1.99 &  2.00  \\\hline
$\|\Delta u - \nabla\cdot G(\nabla u_h)\|$ &  1.64494 &  0.80249 &  0.3989  &  0.19915\\\hline
Order & -- &  1.04 &  1.01 &  1.00  \\\hline
\end{tabular}
\end{table}

\begin{table}[!htp]%\tabcolsep0.1in
\caption{Example \ref{exm1}, CVDT mesh, errors and convergence rates}\label{tabexm12}
\begin{tabular}[c]{|c|c|c|c|c|}
\hline
Dof &   499 &  1920 &  7566 & 29952\\\hline
$\| u - u_h\|$ &  0.01205 &  0.00296 &  0.00069 &  0.00017\\\hline
Order & -- &  2.03 &  2.10  &  2.03\\\hline
$\|\nabla u - \nabla u_h\|$ &  0.37264 &  0.17001 &  0.07189 &  0.03457\\\hline
Order & -- &  1.13 &  1.24 &  1.06\\\hline
$\|\nabla u - G(\nabla u_h)\|$ &  0.03859 &  0.00981 &  0.00234 &  0.00059\\\hline
Order & -- &  1.98 &  2.07 &  1.99\\\hline
$\|\Delta u - \nabla\cdot G(\nabla u_h)\|$ &  1.34359 &  0.65669 &  0.34152 &  0.17329\\\hline
Order & -- &  1.03 &  0.94 &  0.98\\\hline
\end{tabular}
\end{table}

%%%%%%%%%%%%%%%%%%%%%%%%%%%%%%

\begin{table}[!htp]%\tabcolsep0.1in
\caption{Example \ref{exm4}, regular mesh, errors and convergence rates}\label{tabexm41}
\begin{tabular}[c]{|c|c|c|c|c|}
\hline
Dof &   441 &  1681 &  6561 & 25921\\\hline
$\| u - u_h\|$ &  0.04686 &  0.01166 &  0.00292 &  0.00073\\\hline
Order & -- &  2.01 &  2.00   &  2.00  \\\hline
$\|\nabla u - \nabla u_h\|$ &  1.02312 &  0.48499 &  0.23852 &  0.11861\\\hline
Order & -- &  1.08 &  1.02 &  1.01\\\hline
$\|\nabla u - G(\nabla u_h)\|$ &  0.18431 &  0.04731 &  0.01207 &  0.00307\\\hline
Order & -- &  1.96 &  1.97 &  1.98\\\hline
$\|\Delta u - \nabla\cdot G(\nabla u_h)\|$ &  4.90213 &  2.47485 &  1.27612 &  0.66004\\\hline
Order & -- &  0.99 &  0.96 &  0.95\\\hline
\end{tabular}
\end{table}

\begin{table}[!htp]%\tabcolsep0.1in
\caption{Example \ref{exm4}, CVDT mesh, errors and convergence rates}\label{tabexm42}
\begin{tabular}[c]{|c|c|c|c|c|}
\hline
Dof &   499 &  1920 &  7566 & 29952\\\hline
$\| u - u_h\|$ &  0.03992 &  0.00983 &  0.00238 &  0.00059\\\hline
Order & -- &  2.02 &  2.05 &  2.01\\\hline
$\|\nabla u - \nabla u_h\|$ &  0.89327 &  0.38386 &  0.16833 &  0.07981\\\hline
Order & -- &  1.22 &  1.19 &  1.08\\\hline
$\|\nabla u - G(\nabla u_h)\|$ &  0.16318 &  0.04231 &  0.01067 &  0.00275\\\hline
Order & -- &  1.95 &  1.99 &  1.96\\\hline
$\|\Delta u - \nabla\cdot G(\nabla u_h)\|$ &  4.59341 &  2.38537 &  1.29143 &  0.67979\\\hline
Order & -- &  0.95 &  0.89 &  0.93\\\hline
\end{tabular}
\end{table}

\begin{example}\label{exm4} 
For the second example, we consider the biharmonic equation with
non-homogeneous boundary condition 

\bq\label{4e2}\left\{
\begin{aligned}
&\Delta^2u(x,y)=f(x,y), \qquad \forall (x,y)\in \Omega=(0, 1)^2, \\
&u=g,\qquad \nabla u\cdot \pmb{n}=h,\quad \rm{on}\ \partial\Omega.
\end{aligned}\right.\eq
We take \[u=\sin(2\pi x)\sin(2\pi y)\] and then the corresponding problem have following type of boundary conditions
\[u|_{\partial \Omega}= 0, \qquad \nabla u\cdot \pmb{n}|_{\partial \Omega}\neq 0.\]
The corresponding right hand side function $f$ is then take  
\[f=64\pi^4\sin(2\pi x)\sin(2\pi y).\]
\end{example}

The numerical results are reported in Table \ref{tabexm41} and Table
\ref{tabexm42}. The results indicate that both $u$ and $\nabla u$ achieve
optimal convergence order, and the recovered gradient $G(\nabla u_h)$ is
superclose to $\nabla u$. These numerical results show that the recovery
based finite element method also converges with optimal rates for the biharmonic
equation with non-homogeneous boundary conditions.  

%%%%%%%%%%%%%%%%%%%%%%%%%%%%%%%%%%%%%%%%%%%%%%%%%%%%%%%%%%%%%5
In the following, we apply the recovery based linear finite element method for
the biharmonic equation with a singular solution. An adaptive algorithm is used
to resolve the singularity. Note that $\nabla\cdot G(\nabla u_h)$ is a
piecewise constant function, and it can be restored to the continuous piecewise
linear space by recovery operator $G$.  Since $G(\nabla\cdot G(\nabla u_h))$ is
better approximation of $\Delta u$ than $\nabla\cdot G(\nabla u_h)$, we can use
\[\|G(\nabla\cdot G(\nabla u_h)) - \nabla\cdot G(\nabla u_h)\|\] 
as an recovery type a posteriori error estimator to guide the mesh refinement. In the adaptive procedure, the D\"{o}rfler marking strategy \cite{Dor96} with bulk parameter $\theta=0.2$ is used for marking the elements to be refined. We present three numerical examples to investigate the performance
of recovery based linear finite element method on the adaptive meshes.

\begin{example}\label{aexm1}
We consider the model problem \eqref{1e1} on a L-shaped domain $\Omega=(-1,
1)^2\setminus ([0,1)\times (-1, 0])$ with the following exact singular solution
\cite{G92}:

\bq\label{4e3}
u(r,\theta) = (r^2\cos^2\theta-1)^2(r^2\sin^2\theta-1)^2r^{(1+\alpha)}g_{\alpha, \omega}(\theta)
\eq
where
$\alpha = 0.544483736782464$ is a noncharacteristic root of $\sin^2(\alpha\omega)=\alpha^2\sin^2\omega$, $\omega=\frac{3\pi}{2}$ and 
\bq\label{4e4}\begin{aligned}
g_{\alpha, \omega}(\theta) = &\left(\frac{1}{\alpha-1}\sin((\alpha-1)\omega) - \frac{1}{\alpha+1}\sin((\alpha+1)\omega)\right)\times\left(\cos((\alpha-1)\theta) - \cos((\alpha+1)\theta)\right) \\
&-\left(\frac{1}{\alpha-1}\sin((\alpha-1)\theta) - \frac{1}{\alpha+1}\sin((\alpha+1)\theta)\right)\times\left(\cos((\alpha-1)\omega) - \cos((\alpha+1)\omega)\right).
\end{aligned}\eq
\end{example}

\begin{figure}
\begin{center}
    \includegraphics[width=0.45\textwidth]{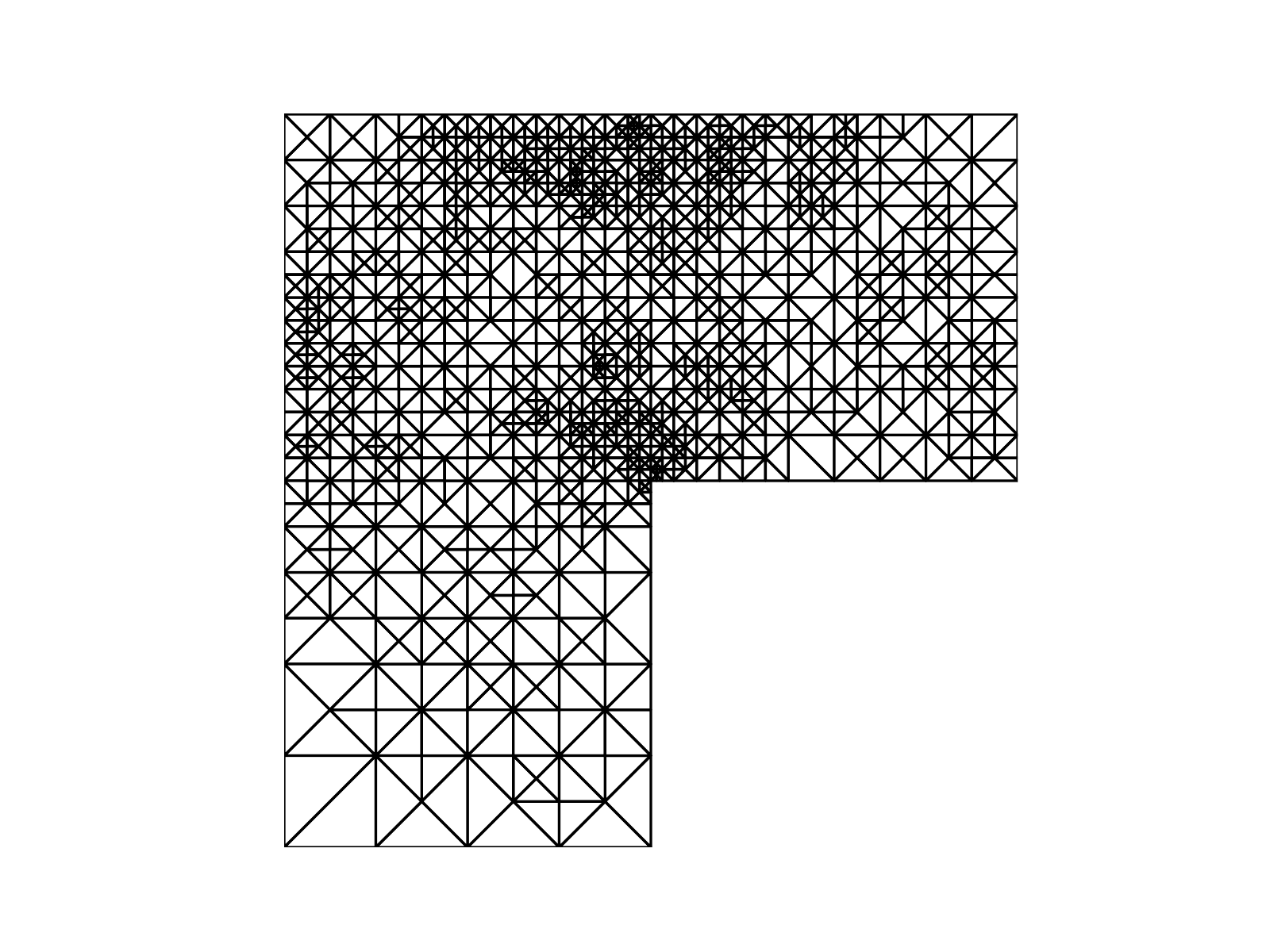}
    \includegraphics[width=0.45\textwidth]{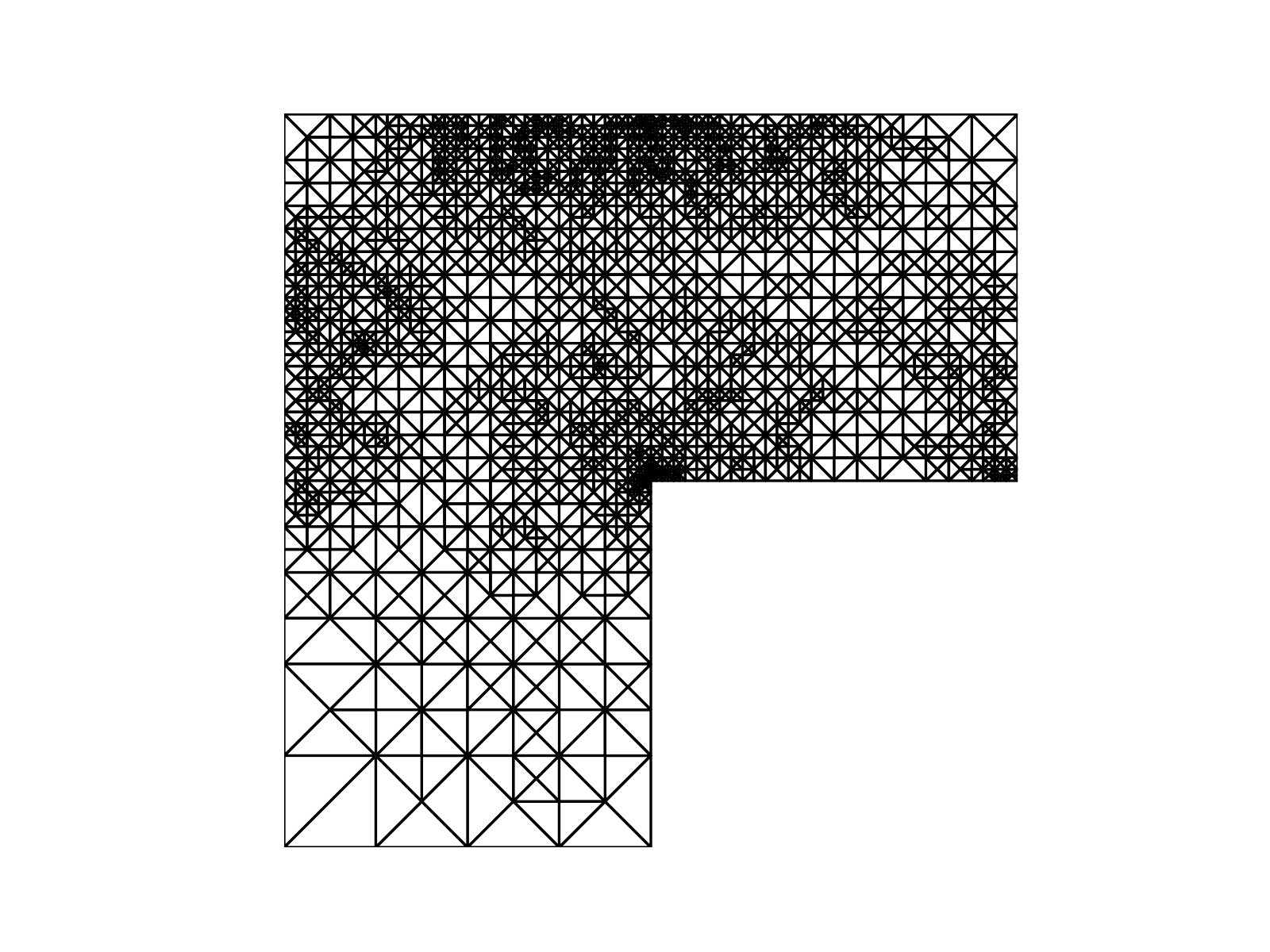}
\caption{Adaptive meshes for Example \ref{aexm1}. Left: level 60; Right: level 70.}
\label{figaexm11}
\end{center}
\end{figure}

\begin{figure}
\begin{center}
    \includegraphics[width=0.5\textwidth]{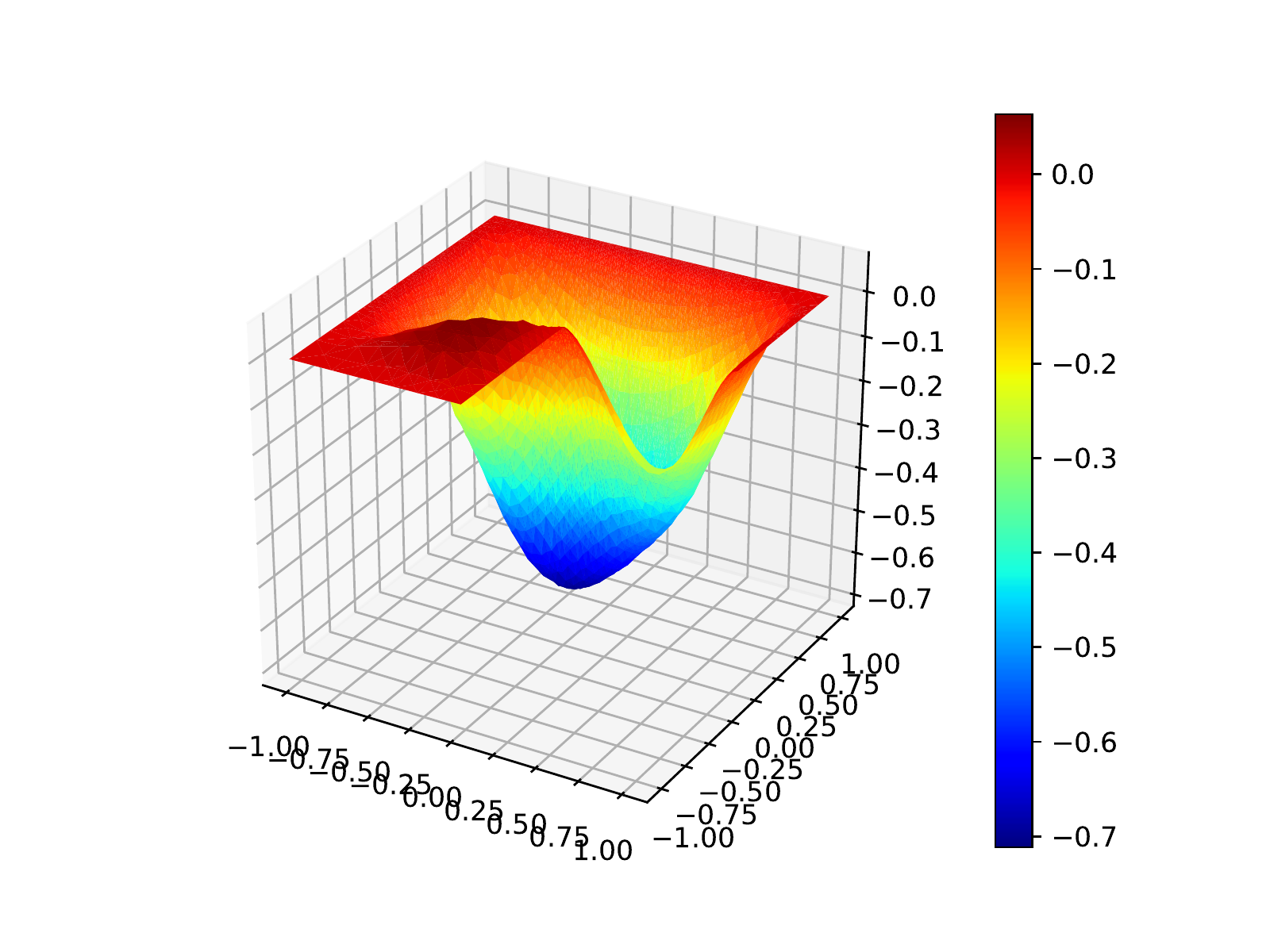}
    \includegraphics[width=0.4\textwidth]{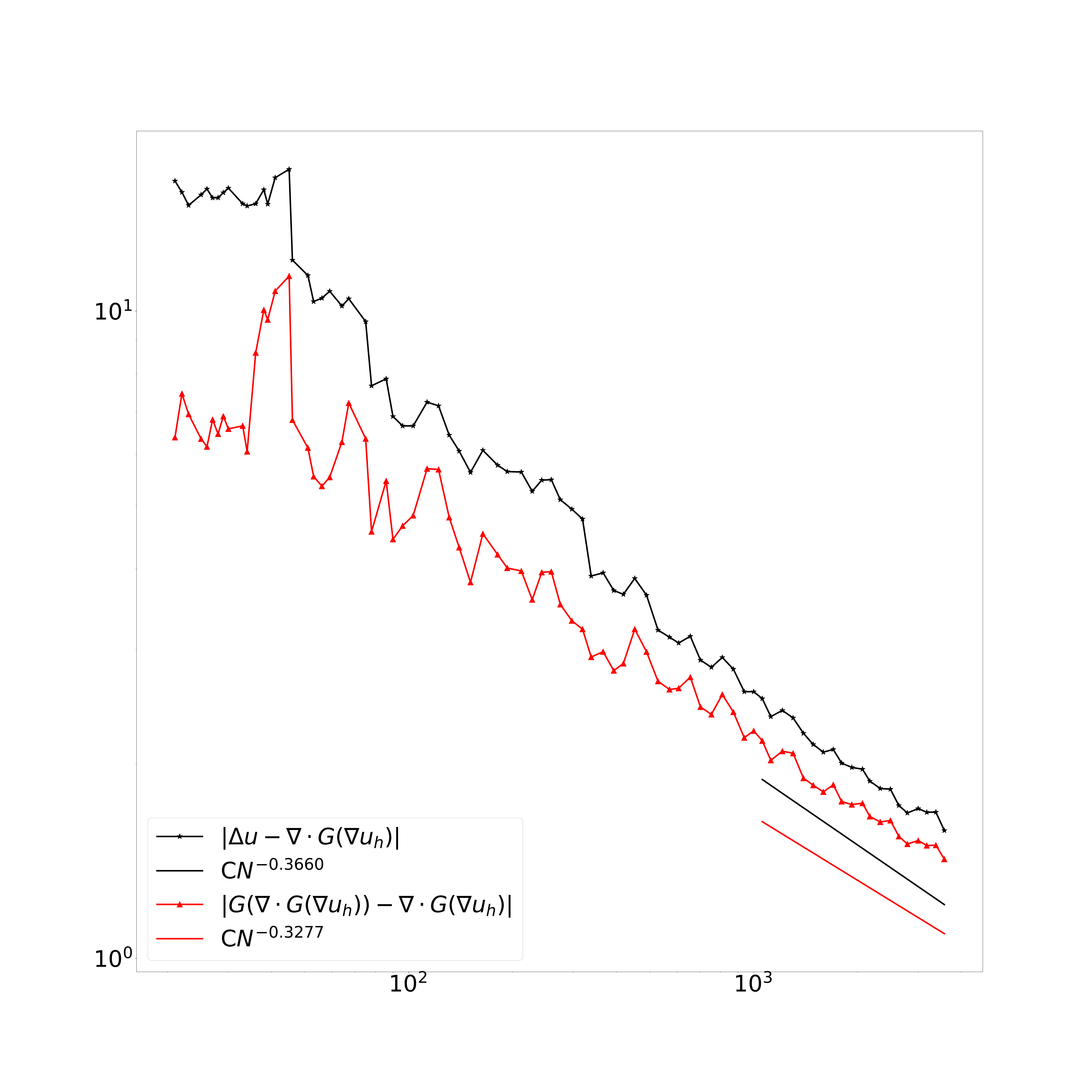}
\caption{Numerical solution and errors history of Example \ref{aexm1}.}
\label{figaexm12}
\end{center}
\end{figure}

Figure \ref{figaexm11} shows the adaptive meshes at refinement level 60 and level 70. 
The error estimator captures the singularities of the solution throughout the mesh refinement process. 
The numerical solution, and corresponding exact errors and error estimators are presented in Figure \ref{figaexm12}. 
We have observed that $\|G(\nabla\cdot G(\nabla
u_h))-\nabla\cdot G(\nabla u_h)\|\approx 1.1\times \|\Delta u-\nabla\cdot G(\nabla u_h)\|$, which means the error estimator is reliable and efficiency.

%%%%%%%%%%%%%%%%%%%%%%%%%%%%%%%%%%%%%%%%%%%%%%%%%%%%%%%

\begin{example}\label{aexm2}
In this example, we take $\Omega:=(-1,1)^2\setminus conv\{(0, 0), (1, -1), (1,
0)\}$. We consider the model problem \eqref{1e1} on $\Omega$ with the exact
solution given by \eqref{4e3} with $\alpha=0.505009698896589,
\omega=\frac{7\pi}{4}$ and $g_{\alpha,\omega}(\theta)$ is of the form
\eqref{4e4}. 
\end{example}

\begin{figure}
\begin{center}
    \includegraphics[width=0.45\textwidth]{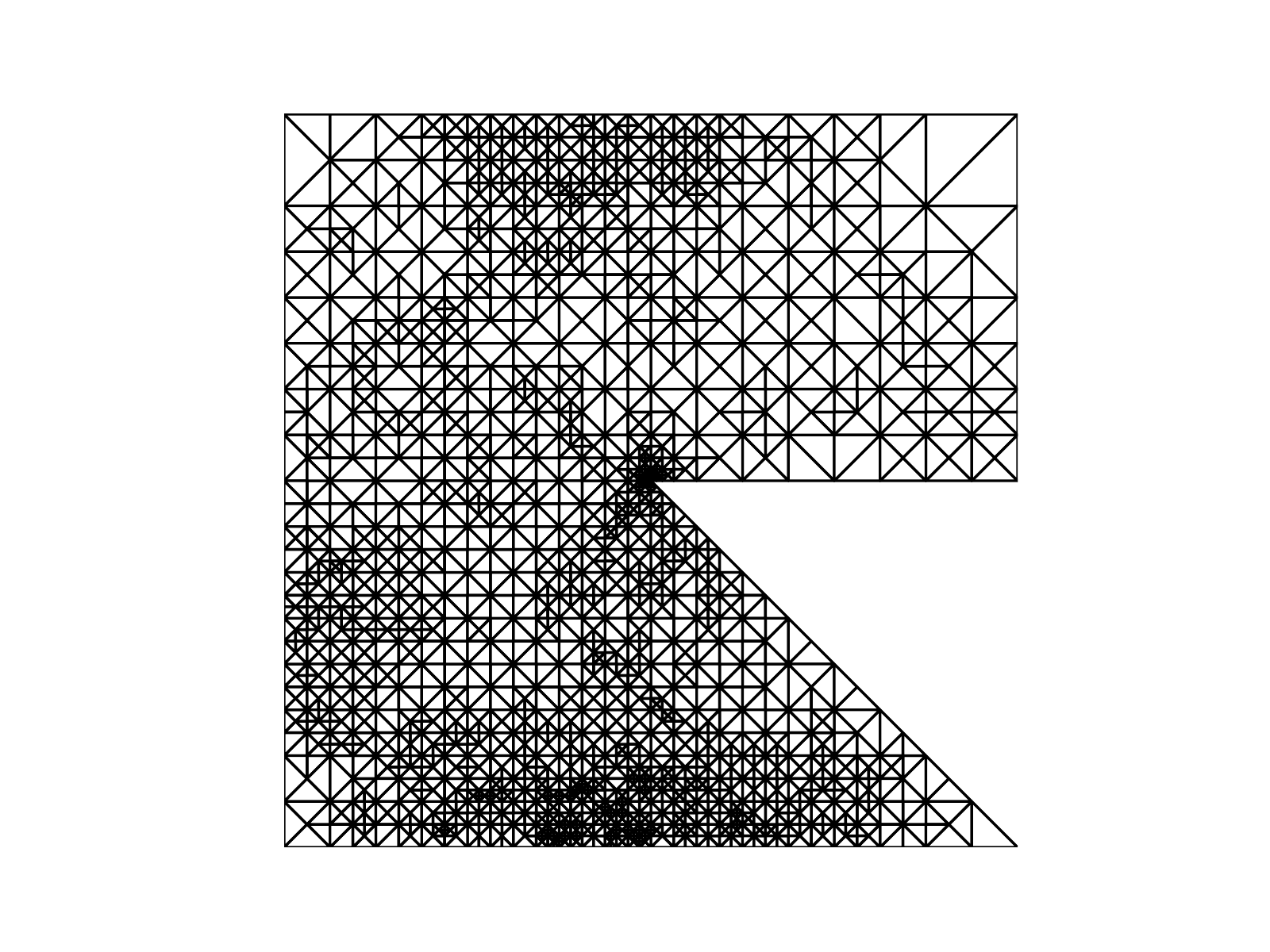}
    \includegraphics[width=0.45\textwidth]{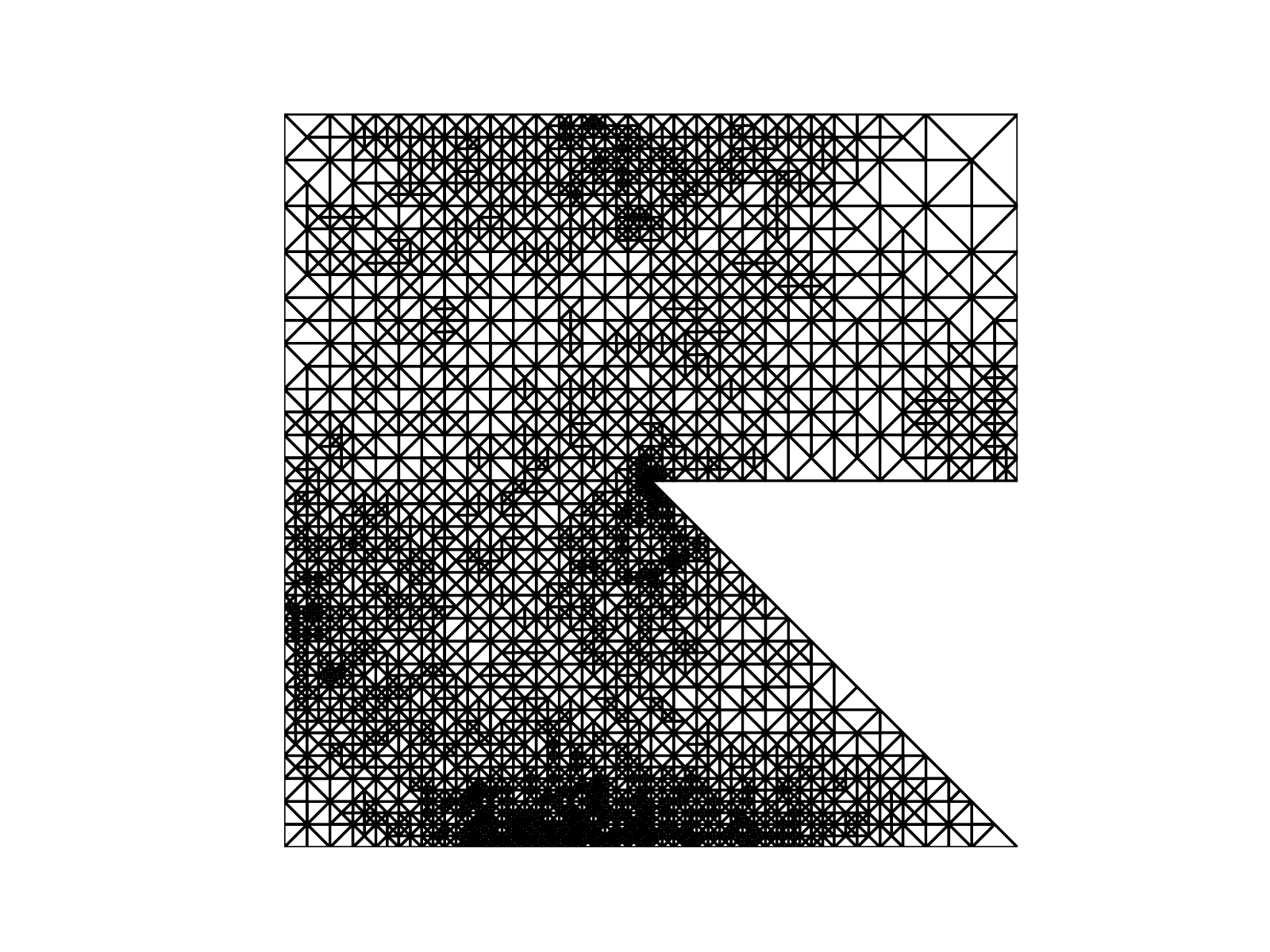}
\caption{Adaptive meshes for Example \ref{aexm2}. Left: level 60; Right: level 70.}
\label{figaexm21}
\end{center}
\end{figure}

\begin{figure}
\begin{center}
    \includegraphics[width=0.5\textwidth]{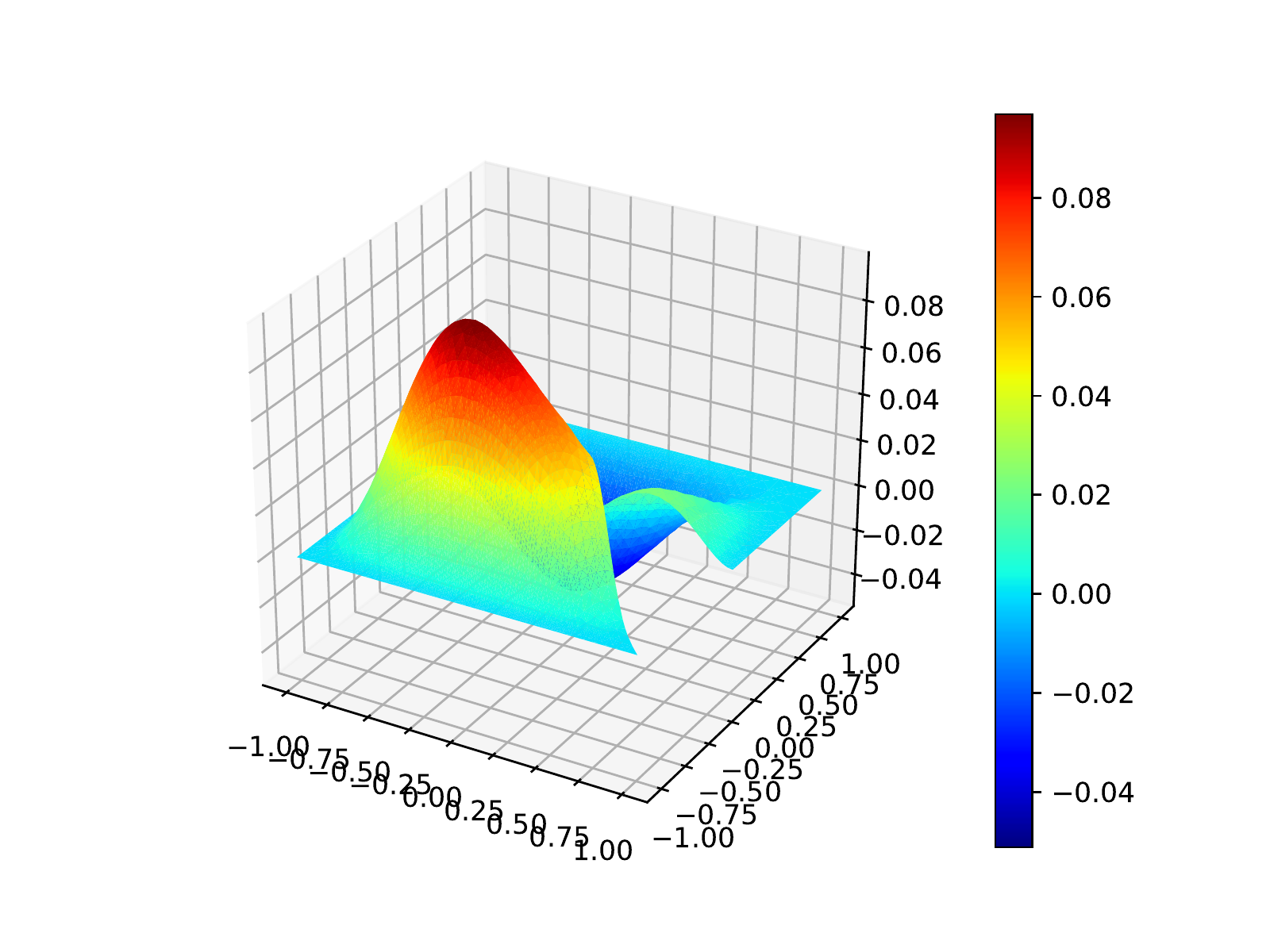}
    \includegraphics[width=0.4\textwidth]{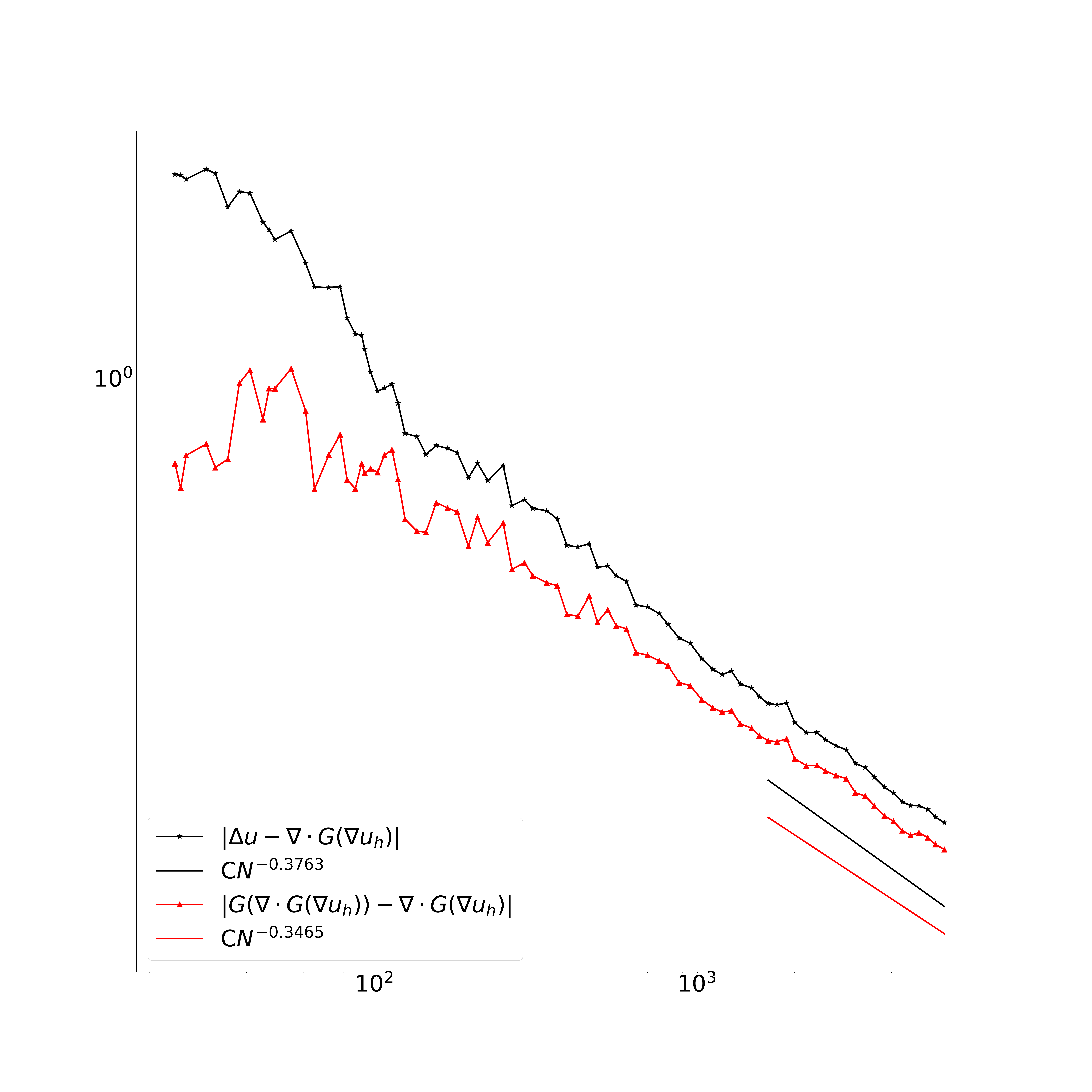}
\caption{Numerical solution and errors history of Example \ref{aexm2}.}
\label{figaexm22}
\end{center}
\end{figure}

Figure \ref{figaexm21} and Figure \ref{figaexm22} display the adaptive meshes, numerical solution and the
convergence history of the error estimators and the exact errors. As in the
previous example, the error estimator yields a good approximation of the true Laplace error, and the singularities of the solution
are well predicted by the error estimator throughout the mesh refinement
process. On the adaptive meshes, we see clearly that the adaptive
mesh-refinement mainly concentrates on the V-corner. We also observe some
additional refinement near the boundary where the gradient is relatively
large. 

%%%%%%%%%%%%%%%%%%%%%%%%%%%%%%%%%%%%%%%%%%%%%%%%%%%%%%%

\begin{example}\label{aexm3}
In this example, we consider the problem \eqref{4e1} with $f=1$ on the
nonconvex domain $\Omega$ with the corners $(0, 0)$, $(1, 0)$, $(1, 1)$, $(2,
0)$, $(3, 0)$, $(1, 2)$, $(3, 4)$, $(2, 4)$, $(1, 3)$, $(1, 4)$ and $(0, 4)$. 
\end{example}

\begin{figure}
\begin{center}
    \includegraphics[width=0.45\textwidth]{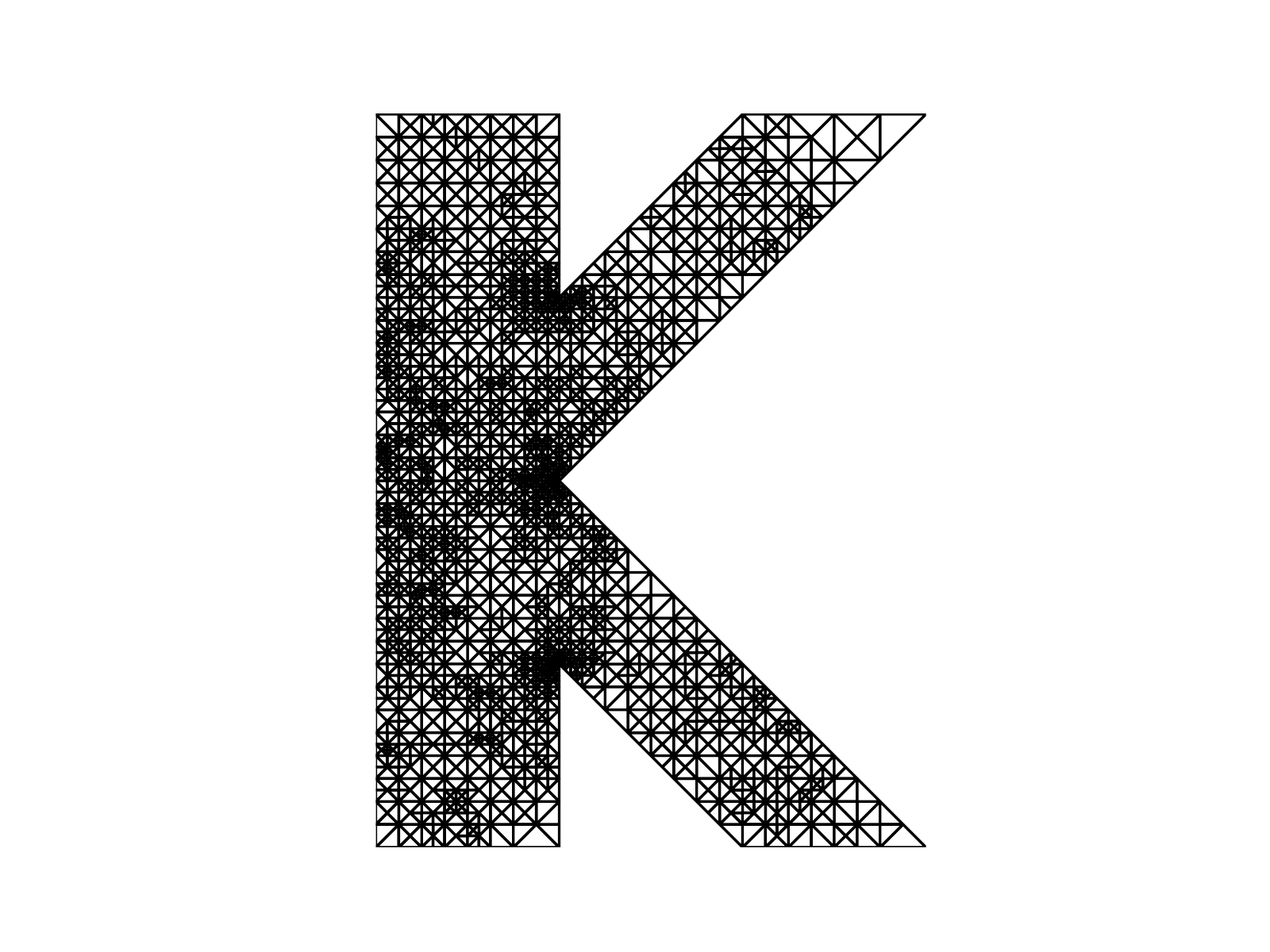}
    \includegraphics[width=0.45\textwidth]{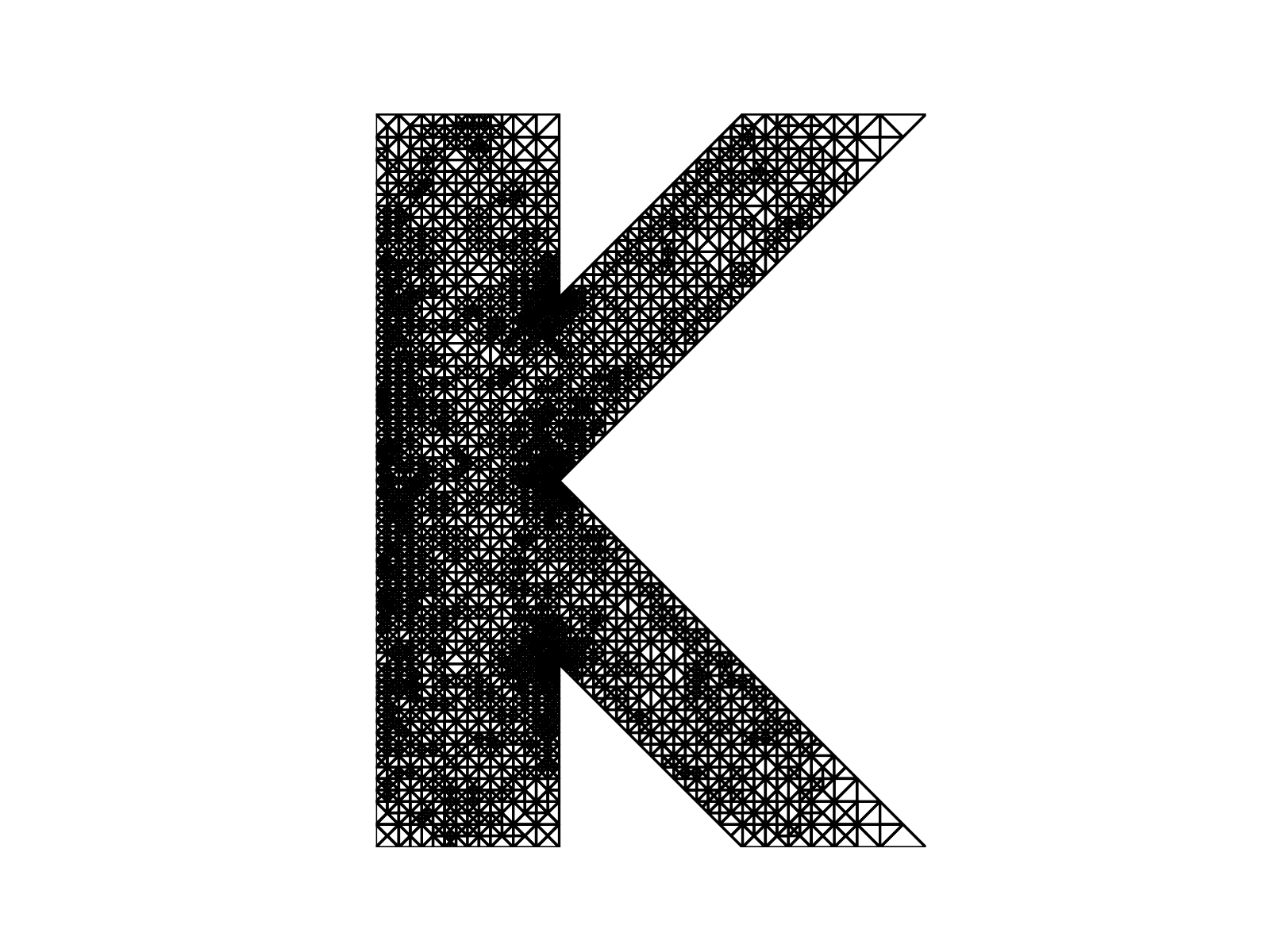}
\caption{Adaptive meshes for Example \ref{aexm3}. Left: level 60; Right: level 70.}
\label{figaexm31}
\end{center}
\end{figure}

\begin{figure}
\begin{center}
    \includegraphics[width=0.5\textwidth]{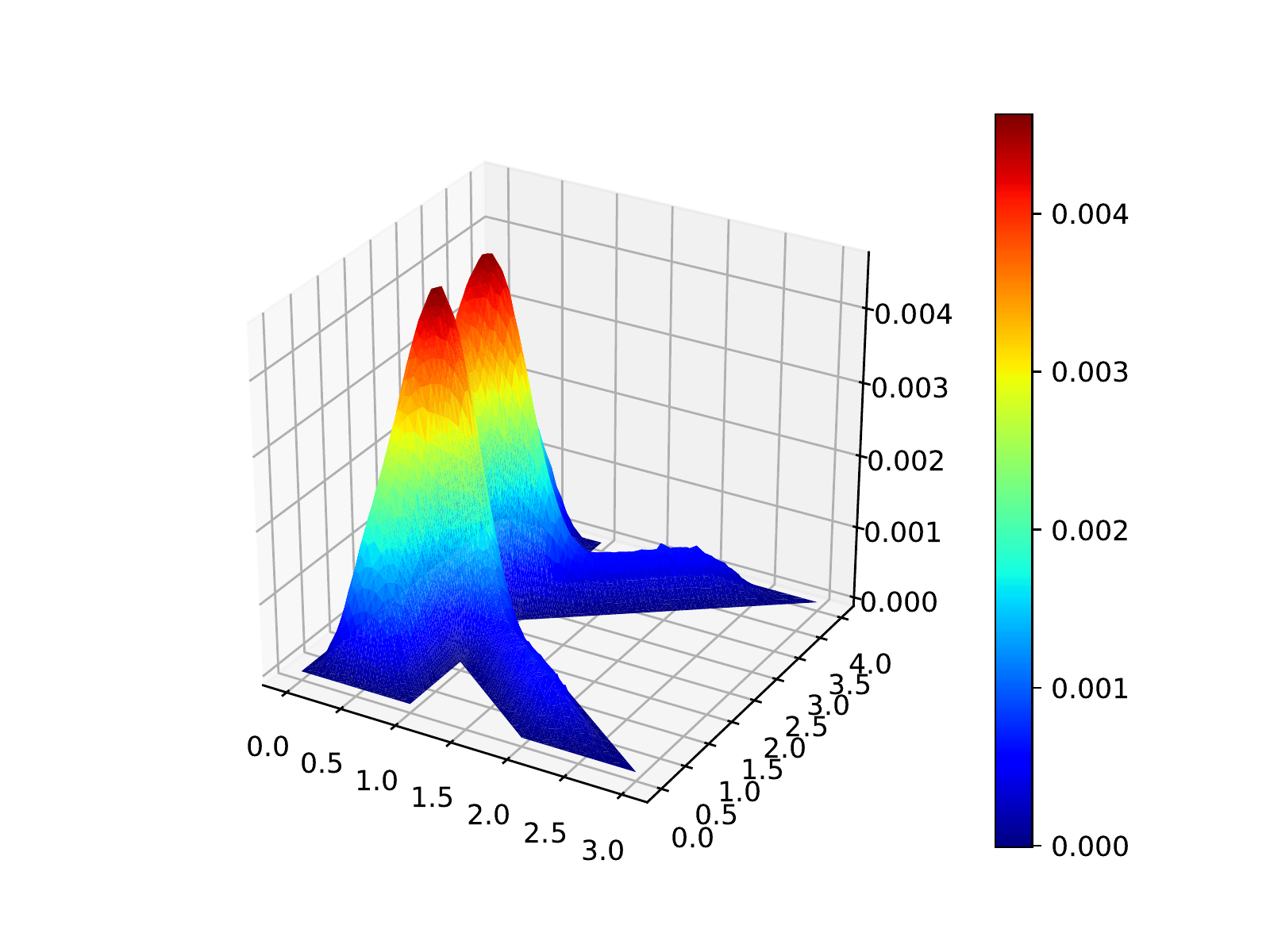}
    \includegraphics[width=0.4\textwidth]{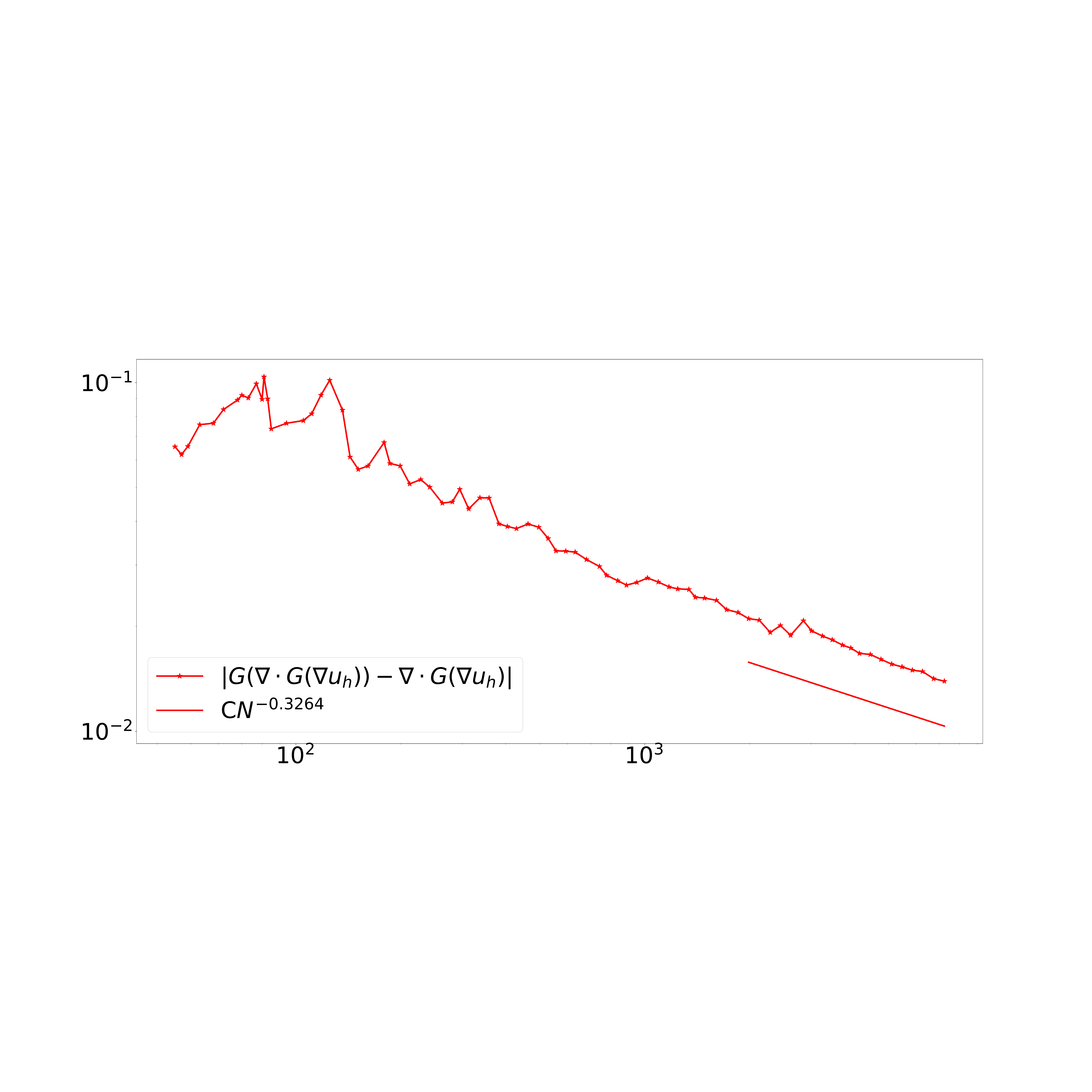}
\caption{Numerical solution and errors history of Example \ref{aexm3}.}
\label{figaexm32}
\end{center}
\end{figure}

In this case, there appear corner singularities at the L-corner and the two
V-corners. Figure \ref{figaexm31} and Figure \ref{figaexm32} plot the adaptive meshes, numerical solution
and the convergence history of the error estimators. We see clearly that the
method finds and clearly distinguishes all the corner singularities and refines
locally near the L-corner and the two V-corners.
%%%%%%%%%%%%%%%%%%%%%%%%%%%%%%%%%%%%%%%%%%%%%%%%%%%%%%%%%%%%%%%

\section{Concluding remarks}\label{sec:CON}
In this paper, we have developed a recovery based linear finite element method
for solving the biharmonic equation. In the discrete weak formulation, the
gradient operator $\nabla$ on the linear finite element space is replaced by
$G(\nabla)$ with $G$ denotes a suitable gradient recovery operator. Thus, the
Laplace opetator $\Delta$ is replaced by $\nabla\cdot G(\nabla u)$.
Furthermore, we impose the boundary condition $\nabla u\cdot n|_{\partial
\Omega}=g_2$ by the boundary penalty method. Numerical examples for the
biharmonic equation with the homogeneous or non-homogeneous boundary conditions
are presented for illustrating the correctness and effectiveness of our method.
They show that the recovery based linear finite element method converges with
optimal rates, and the recovered gradient is superclose to the exact one. We
also numerical investigate the effectiveness of the recovery based finite
element method on adaptive meshes. 
The results show that the error estimator captures the singularities of the solution throughout the mesh refinement process.

For the recovery based finite element method for high order partial
differential equations, we will continuou our works in the following issues:
i) design efficient implementation of the recovery based finite element method,
which incorporate with other gradient recovery operators besides the weighted
averaging method; ii) derive the error estimation for recovery based finite
element method; iii) design the preconditioner for the linear algebra system
which is resulting from the recovery based fintie element method; iv) extend
the recovery based finite element method for other high order partial
differential equation, such as the fourth order parabolic equation and the
Cahn-Hilliard type equation arising from the phase filed models. We will report
these results and applications in our future works.
%%%%%%%%%%%%%%%%%%%%%%%%%%%%%%%%%%%%

\section*{Acknowledgments}
Huang's research was partially supported by NSFC Project (91430213). 
Wei's research was partially supported by Hunan
Provincial Civil-Military Integration Industrial Development Project. 
Yang's research was partially supported by NSFC Project (11771371) and Hunan Education Department Project (15B236). 
Yi's research was partially supported by NSFC Project (11671341), Hunan Provincial NSF Project (2015JJ2145) and Hunan Education Department Project (16A206). 

%%%%%%%%%%%%%%%%%%%%%%%%%%%%%%%%%%%


\begin{thebibliography}{10}

\bibitem{BG11}
   E.M. Behrens and J. Guzman, 
 \newblock  A mixed method for the biharmonic problem based on a system of first-order equations. 
 \newblock {\em  SIAM J. Numer. Anal.}, 49:789-817, 2011.  

\bibitem{BCCF09}
   M. Ben-Artzi, I. Chorev, J.-P. Croisille and D. Fishelov,
 \newblock A compact difference scheme for the biharmonic equation in planar irregular domains. 
 \newblock {\em  SIAM J. Numer. Anal.}, 47:3087-3108, 2009. 

\bibitem{BL04}
   C. Bi and L. Li,
 \newblock Mortar finite volume method with Adini element for biharmonic problem. 
 \newblock {\em  J. Comput. Math.}, 22:475-488, 2004.    

\bibitem{BR80}
   H. Blum and R. Rannacher,
 \newblock On the boundary value problem of the biharmonic operator on domains with angular corners. 
 \newblock {\em  Math. Mech. Appl. Sci.}, 2:556-581, 1980.     

\bibitem{BPB02}
   J. H. Bramble, J. E. Pasciak and C. Bacuta,
 \newblock Shift theorems for the biharmonic Dirichlet problem. 
 \newblock {\em  In Recent Progress in Computational and Applied PDEs}, 1-26, New York, Kluwer Academic/Plenum Publishers, 2002.      

\bibitem{BS05}
  S. Brenner and L. Sung, 
 \newblock  $C^0$ interior penalty methods for fourth order elliptic boundary value problems on polygonal domains. 
 \newblock {\em J. Sci. Comput.}, 22:83-118, 2005.  

\bibitem{BF91}
  F. Brezzi and M. Fortin, 
 \newblock  Mixed and hybrid finite element methods. 
 \newblock {\em Spring-Verlag, New York}, 1991.   

\bibitem{CGZZ16}
 H. Chen, H. Guo, Z. Zhang and Q. Zou,  
 \newblock A $C^0$ finite element method for two fourth-order eigenvalue problems. 
 \newblock {\em IMA J. Numer. Anal.}, DOI: https://doi.org/10.1093/imanum/drw051.   

\bibitem{CLL08}
  G. Chen, Z. Li and P. Lin, 
 \newblock  A fast finite difference method for biharmonic equations on irregular domains and its appication to an incompressible Stokes flow. 
 \newblock {\em Adv. Comput. Math.}, 29:113-133, 2008.  

\bibitem{C91}
 P. Ciarlet,  
 \newblock The finite element method. In P.G. Ciarlet and J.-L. Lions, editors, Part I, Handbook of Numerical Analysis, III.
 \newblock {\em North-Holland, Amsterdam}, 1991.   

\bibitem{Dor96}
   W. D\"{o}rfler,
 \newblock  A convergent adaptive algorithm for Poisson's equation. 
 \newblock {\em  SIAM J. Numer. Anal.}, 33: 1106-1124, 1996.

\bibitem{E71}
   L.W. Ehrlich,
 \newblock  Solving the biharmonic equation as coupled finite difference equations. 
 \newblock {\em  SIAM J. Numer. Anal.}, 8: 278-287, 1971.   

\bibitem{EGHLMT02}
   G. Engel, K. Garikipati, T. J. R. Hughes, M. G. Larson, L. Mazzei and R. L. Taylor,
 \newblock  Contimuous/discontinuous finite element approximations of fourth-order elliptic problems in structural and continuum mechanics with applications to thin beams and plates, and strain gradient elasticity. 
 \newblock {\em Comput. Methods Appl. Mech. Engrg.}, 191:3669-3750, 2002.   

\bibitem{EGHL12}
   R. Eymard, T. Gallouet, R. Herbin and A. Linke,
 \newblock  Finite volume schemes for the biharmonic problem on general meshes. 
 \newblock {\em  Math. Comput.}, 280:2019-2048, 2012.   

\bibitem{GH09}
E.H. Georgoulis and P. Houston,   
 \newblock Discontinuous Galerkin methods for the biharmonic problem. 
 \newblock {\em IMA J. Numer. Anal.}, 29:573-594, 2009.   

\bibitem{GP79}
R. Glowinski and O. Pironneau,   
 \newblock Numerical methods for the first biharmonic equation and for the two-dimensional Stokes problem. 
 \newblock {\em SIAM Rev.}, 21:167-212, 1979.   

\bibitem{G92}
P. Grisvard,   
 \newblock Singularities in Boundary Value Problems, in: Recherches en Mathematiques Appliquees (Research in Applied Mathematics). 
 Vol 22, Masson, Paris, 1992.   

\bibitem{GNP08}
T. Gudi, N. Nataraj and A.K. Pani,   
 \newblock Mixed discontinuous Galerkin finite element method for the biharmonic equation. 
 \newblock {\em J. Sci. Comput.}, 37:139-161, 2008.   

\bibitem{GZZ17}
 H. Guo, Z. Zhang and Q. Zou,  
 \newblock A $C^0$ finite element method for biharmonic problems. 
 \newblock {\em J. Sci. Comput.}, DOI 10.1007/s10915-017-0501-0.   

\bibitem{HJY12}
   Y. Huang, K. Jiang and N. Yi,
 \newblock  Some weighted averaging methods for gradient recovery. 
 \newblock {\em Adv. Appl. Math. Mech.}, 4:131-155, 2012.  

\bibitem{L15}
 B. Lamichhane,  
 \newblock A finite element method for a biharmonic equation based on gradient recovery operators. 
 \newblock {\em BIT Numer. Math.}, 235:5188-5197, 2015. 

\bibitem{M74}
   J.W. McLaurin,
 \newblock  A general coupled equation approach for solving the biharmonic boundary value problem. 
 \newblock {\em  SIAM J. Numer. Anal.}, 11:14-33, 1974.   

\bibitem{SM07}
   E. Suli and I. Mozolevski,
 \newblock  hp-version interior penalty DGFEMs for the biharmonic equation. 
 \newblock {\em Comput. Methods Appl. Mech. Engrg.},  196:1851-1863, 2007. 

\bibitem{W04}
   T. Wang,
 \newblock A mixed finite volume element method based on retangular mesh for biharmonic equations. 
 \newblock {\em  J. Comput. Appl. Math.}, 172:117-130, 2004.  

\bibitem{WX06}
   M. Wang and J. Xu,
 \newblock The Morley element for fourth order elliptic equations in any dimensions. 
 \newblock {\em  Numer. Math.}, 103:155-169, 2006.   

\end{thebibliography}
\end{document}